\documentclass{amsart}      %amsart                  %[12pt,leqno,a4paper]
\usepackage{amssymb,amsmath,eufrak,enumerate,color}
\usepackage[dvips]{graphicx}
\usepackage[english]{babel}                   % deleted russian
\newcommand{\Aut}[1]{\mathrm{Aut}\left(#1\right)}     %  Aut(G)
\newcommand{\Hol}[1]{\mathrm{Hol}\left(#1\right)}     %  Hol(G)
\newcommand{\Inn}[1]{\mathrm{Inn}\left(#1\right)}     %  Inn(G)
\newcommand{\Out}[1]{\mathrm{Out}\left(#1\right)}     %  Out(G)
\newcommand{\Outdiag}[1]{\mathrm{Outdiag}\left(#1\right)}  %  Outdiag(G)
\newcommand{\abs}[1]{\left\vert#1\right\vert}         %  valore assoluto
\newcommand{\set}[1]{\left\{#1\right\}}               %  insieme
\newcommand{\seq}[1]{\left<#1\right>}                 %  sottogr generato
\newcommand{\lto}{\longrightarrow}

\newcommand{\mto}{\mapsto}
\newcommand{\lmto}{\longmapsto}
\newcommand{\n}{\noindent}
\newcommand{\PSL}{\mathrm{L}}           
\newcommand{\PSU}{\mathrm{U}}

\newcommand{\EE}{\mathrm{E}}

\newcommand{\wt}[1]{\widetilde{#1}}
\newtheorem{thm}{Theorem}%[section]
\newtheorem*{namedthm*}{\namedthmname}

\newtheorem{lem}{Lemma}%[section]
\newtheorem{cor}{Corollary}%[section]
\newtheorem{pro}{Proposition}%[section]
\newtheorem{rem}{Remark}%[section]
\begin{document}
\title[On the holomorph of finite semisimple 
groups]{On the holomorph  of finite semisimple groups} 
\date{\today}
 \author{Russell D.~Blyth}
 \address{Department of Mathematics and Computer Science, Saint Louis University, 220 N. Grand Blvd., St. Louis, MO 63103, USA.}
 \email{russell.blyth@slu.edu}
 \author{Francesco Fumagalli}
 \address{Dipartimento di Matematica e Informatica ``Ulisse Dini'',
         viale Morgagni 67/A, 50134 Firenze, Italy.}
 \email{francesco.fumagalli@unifi.it}

% \thanks{}
\begin{abstract}
Given a finite nonabelian semisimple group $G$, we describe 
those groups that have the same holomorph as $G$, that is, 
those regular subgroups $N\simeq G$ of $S(G)$, the group of
permutations on the set $G$, such that 
$N_{S(G)}(N)=N_{S(G)}(\rho(G))$, where $\rho$ is
the right regular representation of $G$. 
\end{abstract}
\maketitle

\section{Introduction}\label{sec:intro}
Let $G$ be any finite group and let $S(G)$ be the symmetric group 
on the set of elements of $G$. We denote by
$\rho: G \to S(G)$ and $\lambda: G \to S(G)$ respectively 
the right and the left regular representations of $G$. 
The normalizer $$\Hol{G}=N_{S(G)}(\rho(G))$$ 
is the \emph{holomorph} 
of $G$, and it is isomorphic to the natural extension of $G$ by its 
automorphism group $\Aut{G}$. It is well-known that 
$\Hol{G}=N_{S(G)}(\lambda(G))$. 

In \cite{Mil} the \emph{multiple holomorph} of $G$ has been defined
as $$\textrm{NHol}(G)=N_{S(G)}(\Hol{G})$$ 
and it is proved that the quotient group 
$$T(G)=\textrm{NHol}(G)/\textrm{Hol}(G)$$ 
acts regularly by conjugation on the set of the regular subgroups $N$ of 
$S(G)$ that are isomorphic to $G$ and have the same holomorph as $G$, 
that is, $T(G)$ acts regularly on the set 
$$\mathcal{H}(G)=\set{H\leq S(G)\vert H\, \textrm{ is regular, }
\, H\simeq G, N_{S(G)}(H)=\Hol{G}}.$$

There has been some attention both in the distant past (see \cite{Mills}) 
and quite recently (see \cite{Kohl}, \cite{CDV} \cite{CDV17} and 
\cite{Car}) on the problem of determining, for $G$ in a given class of 
groups, those groups that have the same holomorph
as $G$ and, in particular, the set $\mathcal{H}(G)$ and the structure of 
the group $T(G)$. 
Recently, in \cite{CDV17} the authors attack this problem when $G$ is a 
perfect group, obtaining complete results for centerless groups (see 
Theorem 7.7 in \cite{CDV17}). 
However they leave open some interesting questions when the center is 
non-trivial. 

The aim of this paper is to completely resolve the case when $G$ is a 
finite semisimple group.

One of the main obstacles for describing the holomorph of a finite semisimple group (see \cite[Remark 7.12]{CDV} and also \cite[ADV - 4B]{AGTA}) 
was to completely classify those finite nonabelian simple groups that admit 
automorphisms acting like inversion on their Schur multiplier. 
In Proposition \ref{pro:aut_quasisimple} we produce 
a complete analysis, whose proof depends on the Classification 
of Finite Simple Groups. It turns out that 
there is a small list $\mathcal{L}$ (see after Proposition 
\ref{pro:aut_quasisimple} for its definition) of related 
quasisimple groups having automorphisms inverting their center. 
Our main result can be therefore stated as follows: \\

%
%\begin{namedthm}{Theorem 0}\label{thm:0}
%Let $G$ be a finite nonabelian semisimple group and let 
%$G=A_1A_2\ldots A_n$ be its unique central decomposition 
%as a product of $\Aut{G}$-indecomposable factors. 
%Assume that the components of $G$ do not belong to $\mathcal{L}$. 
%Then the set $\mathcal{H}(G)$ is completely known and 
%$T(G)$ is an elementary abelian $2$-group of rank $n$.
%\end{namedthm}
%

\noindent
{\bf Theorem. }{\it 
Let $G$ be a finite nonabelian semisimple group and let 
$$G=A_1A_2\ldots A_n$$ be its unique central decomposition 
as a product of $\Aut{G}$-indecomposable factors. 
Assume that the number of factors $A_i$ of $G$ having components 
in $\mathcal{L}$ is exactly $l$, for some $0\leq l\leq n$. 
Then the set $\mathcal{H}(G)$ is completely known and it has cardinality
$\abs{\mathcal{H}(G)}= 2^h,$ 
where $\min\set{n-l+1,n}\leq h\leq n$.\\ 
Also, the group $T(G)$ is an elementary abelian group of order $2^h$.\\
Moreover, if the centers of the factors $A_i$ are all 
amalgamated, then $\abs{\mathcal{H}(G)}= 2^m$ and $T(G)$ is elementary abelian 
of order $2^m$, where $m=\min\set{n,n-l+1}$.}

%(Dire da qualche parte)
%in this context we implicitly identify 
%$G$ with its copy $\rho(G)$ inside $S(G)$
%
%\begin{align*}
%\rho: G & \lto S(G)            &  \lambda: G & \lto S(G)\\ 
%      g & \lmto (x\lmto xg) &           g & \lmto (x\lmto gx) 
%\end{align*}

\section{Semisimple groups}\label{sec:J_semisimple}
To establish the notation, note that we write permutations
as exponents, and denote compositions of maps by
juxtaposition. We compose maps left-to-right.
We consider the right and the left regular ``representations" of $G$, defined by 
\begin{align*}
\rho: G &\lto S(G)    &  \lambda: G &\lto S(G)   \\
    g &\lmto (x \mto xg) & g &\lmto (x \mto gx).
\end{align*}
\begin{rem}
Since composition of maps is left-to-right, 
with our definition the map $\lambda$ is an \emph{antihomomorphism}, 
not a homomorphism, from $G$ to $S(G)$.
We have chosen this definition, over the standard one 
(where $\lambda(g)$ maps $x\in G$ to $g^{-1}x$) for the same reasons 
as in \cite{CDV}.
\end{rem}

The first proposition recalls some basic facts (see 
\cite[Proposition 2.4]{CDV17}). 
The proof is left to the reader.
\begin{pro}\label{pro:basic}
Let $G$ be any group and let $\mathrm{inv}$ be the inversion map on $G$ 
defined by $\mathrm{inv}(g)=g^{-1}$ for every $g\in G$. 
The following hold:
\begin{enumerate}
\item $C_{S(G)}(\rho(G))=\lambda(G)$ and $C_{S(G)}(\lambda(G))=\rho(G)$.
%\item The stabilizer of $1$ in $N_{S(G)}(\rho(G))$ is $\Aut{G}$. 
%{\bf{[maybe delete]}} 
\item $N_{S(G)}(\rho(G))=\Aut{G}\ltimes\rho(G)=\Aut{G}\ltimes
\lambda(G)=N_{S(G)}(\lambda(G))$.
\item $\rho(g)^{\mathrm{inv}}=\lambda(g^{-1})$ and 
$\lambda(g)^{\mathrm{inv}}=\rho(g^{-1})$ for every $g\in G$. 
In particular, $\mathrm{inv}$ conjugates $\rho(G)$ to $\lambda(G)$ 
(and vice versa), and it centralizes $\Aut{G}$; therefore 
$\mathrm{inv}$ normalizes $N_{S(G)}(\rho(G))$, that is, 
$\mathrm{inv}\in \mathrm{NHol}(G)$.
\end{enumerate}
\end{pro}

Recall that a \emph{quasisimple group} is a perfect group $X$ such that 
$X/Z(X)$ is simple, and that a \emph{semisimple group} is a central 
product of quasisimple groups, that is, a group $X=X_1X_2\ldots X_t$ 
with each $X_i$ quasisimple and such that $[X_i,X_j]=1$ for every 
$i\ne j$. % (see definition on page 16 in \cite{GLS1}). 
The quasisimple normal subgroups 
$X_i$ of $X$ are called the \emph{components} of $X$.
Note in particular that semisimple groups are perfect. 

Every finite semisimple group admits a unique 
decomposition as a central product of characteristic subgroups.
\begin{lem}\label{lem:0}
Let $G$ be a finite semisimple group. Then $G$ is a central product of 
perfect centrally indecomposable $\Aut{G}$-subgroups:
$$G=A_1A_2 \ldots A_n.$$
Moreover, the integer $n$ and the subgroups $A_i$ (for $i=1,2,\ldots, n$)
are uniquely determined 
(up to permutation).
\end{lem}
\begin{proof}
Consider the Remak-Krull-Schmidt decomposition of 
$\Inn{G}\simeq G/Z(G)$ as an $\Aut{G}$-group, and let this be
\begin{equation}\label{eq:G_central_1}
\frac{G}{Z(G)}=\frac{M_1}{Z(G)}\times \frac{M_2}{Z(G)}\times \ldots
\times \frac{M_n}{Z(G)}.
\end{equation}
Since $G$ is perfect, each $M_i/Z(G)$ is perfect. 
In particular, each $M_i$ is equal to $M_i'Z(G)$, where $M_i'$ is perfect. Now for every 
$j\neq i$ we have that
$$[M_j,M_i']=[M_j,M_i]\leq Z(G),$$
and thus $M_j$ induces by conjugation a central automorphism on $M_i'$. 
Since perfect groups have no nontrivial central automorphisms
(see \cite{Hug}, or \cite[Lemma 7.1]{CDV}), we have $[M_i,M_j]=1$ for every $j\neq i$. In particular, 
equation (\ref{eq:G_central_1}) and the fact that $G$ is perfect imply the following central factorization of $G$:
\begin{equation}\label{eq:G_central}
G=A_1A_2\ldots A_n,
\end{equation}
where $A_i=M_i'$, for each $i=1,2,\ldots, n$. Note that 
the $A_i$ are perfect $\Aut{G}$-subgroups, which are indecomposable as 
$\Aut{G}$-subgroups. 
Finally the uniquiness of the factorization (\ref{eq:G_central_1}) (see 
\cite[3.3.8]{Rob}) and, again, the fact that perfect groups have no nontrivial 
central automorphisms, imply that the central product decomposition 
(\ref{eq:G_central}) is also unique. 
%(Recall also that the Krull–Remak–Schmidt of G in terms of indecomposable
%Aut(G)-subgroups is unique, because of [18, 3.3.8, p. 83] and Lemma 7.1.(2).)
\end{proof}

We make use of the same notation as \cite{CDV}.
In particular, we define the following subsets of subgroups of $S(G)$:
\begin{align*} 
\mathcal{H}(G)&=\set{H\leq S(G)\vert H\, \textrm{ is regular, }\, H\simeq G, N_{S(G)}(H)=\Hol{G}},\\ 
\mathcal{I}(G)&=\set{H\leq S(G)\vert H\, \textrm{ is regular, }\, N_{S(G)}(H)=\Hol{G}}, \\
\mathcal{J}(G)&=\set{H\leq S(G)\vert H\, \textrm{ is regular, }\, N_{S(G)}(H)\geq \Hol{G}}.
\end{align*}
Note that $\mathcal{H}(G)\subseteq\mathcal{I}(G)\subseteq\mathcal{J}(G)$.\\

%Also let $\textrm{NHol}(G)=N_{S(G)}(\Hol{G})$ be the \emph{multiple holomorph} of $G$ and $T(G)$ the quotient group
%$$T(G)=\textrm{NHol}(G)/\textrm{Hol}(G).$$ 
%Note that $T(G)$ acts regularly by conjugation on $\mathcal{H}(G)$. 

From now on we assume that $G$ is a finite semisimple group and we 
write $G$ as a central product of indecomposable $\Aut{G}$-subgroups, 
in a unique way (by Lemma~\ref{lem:0}):
$$G=A_1A_2\ldots A_n.$$
Note that by Proposition \ref{pro:basic} (1), we have that
$[\rho(A_i),\lambda(A_j)]=1$ for every $i\neq j$.\\
Fix $I$ to be the set $\set{1,2,\ldots,n}$.
For each subset $J$ of $I$ we denote the central product 
$\prod_{j\in J}A_j$ by $A_J$. Then for each subset $J$ of $I$ we may
define the subgroup $G_J$ of $S(G)$ to be 
$$G_J=\rho(A_J)\lambda(A_{J^c}),$$
where $J^c=I\setminus J$. Note that $G_I=\rho(G)$ and 
$G_{\varnothing}=\lambda(G)$.

\begin{lem}\label{lem:G_J} Each $G_J$ is a subgroup of $S(G)$ that 
lies in $\Hol{G}$.\end{lem}
\begin{proof}
Since $[\rho(G),\lambda(G)]=1$,  
it follows for every $x_J,y_J\in A_J$, $x_{J^c},y_{J^c}\in A_{J^c}$ that
\begin{align*}
\rho(x_J)\lambda(x_{J^c})\big(\rho(y_J)\lambda(y_{J^c})\big)^{-1} &=
\rho(x_J)\rho(y_J)^{-1}\lambda(x_{J^c})\lambda(y_{J^c})^{-1}\\
& =\rho(x_Jy_J^{-1})\lambda(y_{J^c}^{-1}x_{J^c}),
\end{align*}
which lies in $G_J$.

Moreover, since $\rho(A_J)\leq \rho(G)$ and 
$\lambda(A_{J^c})\leq \lambda(G)$, we have 
$$G_J=\rho(A_J)\lambda(A_{J^c})\leq \seq{\rho(G),\lambda(G)}\leq 
\Hol{G},$$
completing the proof of the Lemma.
\end{proof}

\begin{lem}\label{lem:inv} The inversion map $\mathrm{inv}$ 
conjugates $G_J$ to $G_{J^c}$, for every $J\subseteq I$.
\end{lem}
\begin{proof}
The lemma is an immediate consequence of Proposition \ref{pro:basic} (3).
\end{proof}

\begin{lem}\label{lem:J} Assume that $G$ is a finite 
semisimple group. Then, with the above notation,
$\mathcal{J}(G)=\set{G_J\vert J\subseteq I}.$
\end{lem}
\begin{proof}
We first show that $G_J$ acts regularly on the set $G$. For an 
arbitrary $g\in G$ write $g=x_Jx_{J^c}$, with $x_J\in A_J$ and 
$x_{J^c}\in A_{J^c}$. 
Then the element $\rho(x_J)\lambda(x_{J^c}^{-1})$ of $G_J$  
sends $1$ to $g$:
$$1^{\rho(x_J)\lambda(x_{J^c})}=x_{J^c}x_J=x_Jx_{J^c}=g,$$
since $[A_J,A_{J^c}]=1$. Moreover the stabilizer of $1$ in $G_J$ 
consists of the elements $\sigma=\rho(x_J)\lambda(x_{J^c})$ such that 
$x_J=x_{J^c}^{-1}\in A_J\cap A_{J^c}\leq Z(G)$, and therefore 
$\sigma$ is the identity.

To show that each $G_J$ is normal in $\Hol{G}$, since 
$[\rho(G),\lambda(G)]=1$, 
it is enough to show that $\Aut{G}$ normalizes every $G_J$. 
Fix $J\subseteq I$, let $\alpha\in\Aut{G}$, and let  
$\rho(x_J)\lambda(x_{J^c})$ be an arbitrary element of $G_J$,
where $x_J\in A_J$ and 
$x_{J^c}\in A_{J^c}$. Then we have that, for every $g\in G$,

$$g^{\alpha^{-1}\rho(x_J)\lambda(x_{J^c})\alpha}=
(x_{J^c}g^{\alpha^{-1}}x_J)^{\alpha}=x_{J^c}^{\alpha}gx_J^{\alpha}=
g^{\rho(x_J^{\alpha})\lambda(x_{J^c}^{\alpha})}.$$
Therefore
$$\big(\rho(x_J)\lambda(x_{J^c})\big)^{\alpha}=
\rho(x_J^{\alpha})\lambda(x_{J^c}^{\alpha}),$$
which lies in $G_J$, since $A_J$ and $A_{J^c}$ are 
characteristic subgroups of $G$. 

Finally by \cite[Theorem 7.8]{CDV}, $\abs{\mathcal{J}(G)}=2^n$, 
and therefore to complete the proof it remains to show that $G_J\neq G_K$ 
whenever $J\neq K$. 
If $G_J=G_K$ then there exists an $i\in I$ for which $G_J$ contains 
both $\rho(A_i)$ and $\lambda(A_i)$. But then the stabilizer of 
$1$ in $G_J$ would contain 
$\set{\rho(x)\lambda(x^{-1})\vert x\in A_i}$, that is, all 
the conjugates of elements of $A_i$. Since $A_i$ is not 
central in $G$, this contradicts the fact that $G_J$ is regular.
\end{proof}

\section{Automorphisms of finite quasisimple groups}\label{sec:quasisimple}
In this section we classify all finite quasisimple groups that admit 
an automorphism acting like inversion on the center. 
This classification, which is used in the proof of our main result
(Theorem A), is proved in Proposition \ref{pro:aut_quasisimple} 
using the Classification of Finite Simple Groups. This result completely 
answers a question posed in \cite[Remark 7.12]{CDV} (see also \cite[ADV - 
4B]{AGTA}), namely whether there are finite quasisimple groups $L$ such 
that $Z(L)$ is not elementary abelian, and such that $\Aut{L}$ does not 
induce inversion on $Z(L)$, or acts trivially on it.

Before stating and proving Proposition \ref{pro:aut_quasisimple}, we 
introduce some notation and terminology related to automorphisms of 
finite nonabelian simple groups of Lie type. 
We refer the interested reader to the first two chapters of \cite{GLS3}.
 
Let $S$ be any finite group of Lie type. Then, 
by \cite[Theorem 30]{Stein1}, any automorphism of $S$ is a product $idf\!g$,
where $i$ is an inner automorphism of $S$ (an element of $\Inn{S}$), 
$d$ is a diagonal automorphism of $S$, $f$ is a field automorphism of $S$ and $g$ is a graph automorphism of $S$. 
The following notation is quite standard. $\textrm{Inndiag}(S)$,  
$\Phi_S$ and $\Gamma_S$ denote, respectively, the group of inner-diagonal 
automorphisms of $S$, of field automorphisms of $S$, and of graph 
automorphisms of $S$. Further, $\Outdiag{S}$ is defined to be
$\textrm{Inndiag}(S)/\Inn{S}$ (see \cite[Definition 2.5.10]{GLS3}). 

\begin{pro}\label{pro:aut_quasisimple}
Let $K$ be a finite quasisimple group. Then there exists an 
automorphism of $K$ that inverts $Z(K)$ if and
only if $K$ is not isomorphic to one of the following groups:
\begin{enumerate}
%\item $\wt{\PSL_3(4)}$, the universal covering of $\PSL_3(4)$;
%\item a covering of $\PSL_3(4)$ having center isomorphic to 
%$Z_2\times Z_2\times Z_3$ or $Z_2\times Z_4\times Z_3$;
\item a covering of $\PSL_3(4)$, with center 
containing $Z_2\times Z_2\times Z_3$;
%\item $\wt{\PSU_4(3)}$, the universal covering of $\PSU_4(3)$;
%\item a covering of $\PSU_4(3)$ having center isomorphic to 
%$Z_{3}\times Z_4$.
\item a covering of $\PSU_4(3)$, with center 
containing $Z_{3}\times Z_4$;
\item $\wt{\PSU_6(2)}$, the universal covering of $\PSU_6(2)$;
\item $\wt{\null^2\EE_6(2)}$, the universal covering of $\null^2\EE_6(2)$.
\end{enumerate}
\end{pro}
\begin{proof}
%If $K$ is a quasisimple group, then $\Out{K}=\Aut{K}/\Inn{K}$ 
%acts on $Z(K)$, thus for 
As is well known (see, for example, \cite{Asch}), any finite 
quasisimple group $K$ is isomorphic to a quotient 
of the universal covering group of its simple quotient $K/Z(K)$. Also, 
$\Aut{K}\simeq \Aut{K/Z(K)}$ (see for instance \cite[Section 33]{Asch}). 
Therefore, for our purposes it is enough to consider the action of 
$\Aut{S}$ on the Schur multiplier $M(S)$ when $S$ varies among all finite 
nonabelian simple groups. In this situation, the outer automorphism group 
$\Out{S}$ acts on $M(S)$, which, by \cite[Section 5, 6-1]{GL1},
is isomorphic to the direct product of two factors of 
relatively prime orders: $M_c(S)$ and $M_e(S)$. 
The actions of $\Out{S}$ on both factors are
completely described in \cite[Theorem 6.3.1 and Theorem 2.5.12]{GLS3} 
for every finite nonabelian simple group $S$. 
In particular, $\Outdiag{S}$ centralises 
$M_c(S)$, and there is an isomorphism of $\Outdiag{S}$ on $M_c(S)$ 
preserving the action of $\Out{S}$.
Note also that if one of the factors $M_c(S)$ or $M_e(S)$ has 
order at most $2$, since they have coprime orders, 
to prove our statement it is enough to see if inversion 
is induced by $\Out{S}$ on the other factor. 
We may therefore consider the two cases: 
\begin{enumerate}
\item $\abs{M_e(S)}\leq 2$, 
\item $\abs{M_e(S)}>2$.
\end{enumerate}
Case (1) $\abs{M_e(S)}\leq 2$.\\
We prove that in this case there is always an automorphism of 
$S$ inverting $M(S)$. As noted above, it is enough 
to consider the action of $\Out{S}$ on $\Outdiag{S}$, which is 
$\Aut{S}$-isomorphic to $M_c(S)$.
But $\Outdiag{S}$ is always inverted by $\Out{S}$, since, 
by \cite[Theorem 2.5.12]{GLS3}, we have that:
\begin{enumerate}
\item[i)]  if $S\in\set{A_m(q), D_{2m+1}(q), \EE_6(q)}$, then $\Outdiag{S}$
is inverted by a graph automorphism (by 2.5.12 (i)),
\item[ii)]  if $S\in\set{\null^2A_m(q), \null^2D_{2m+1}(q), \null^2\EE_6(q)}$, 
then $\Outdiag{S}$ is inverted by a field automorphism (by 2.5.12 (g)),
\item[iii)] in all other cases, $\Outdiag{S}$ is either trivial or an elementary abelian $2$-group and therefore it is inverted by the trivial automorphism. 
\end{enumerate}

\n
Case (2) $\abs{M_e(S)}>2$.\\
From \cite[Table 6.3.1]{GLS3} and the knowledge of the 
corresponding factor $M_c(S)$ (\cite[Theorem 2.5.12]{GLS3}), we 
can see that if $S$ is not isomorphic to one of the following 
simple groups
$$\PSL_3(4), \PSU_6(2), \null^2\EE_6(2), \PSU_4(3),$$ 
then there exists an automorphism of $S$ that inverts $M(S)$ 
and, therefore, any quasisimple group $K$ such that $K/Z(K)\simeq S$  
admits an automorphism inverting its center. 
We now consider separately the four special cases listed above. 

Let $S=\PSL_3(4)$. Then $M_e(S)\simeq Z_4\times Z_4$ and 
$M_c(S)\simeq Z_3$. Here $\Out{S}=\Sigma\times \seq{u}$, 
with $\Sigma=\mathrm{Outdiag}(S)\Gamma_S \simeq S_3$ and $u$ 
the image in $\Out{S}$ of a graph-field automorphism of order 2.
By \cite[Theorem 6.3.1]{GLS3} $u$ is the only element of $\Out{S}$ 
that induces inversion of $M_e(S)$. Now, $u$ is $\Aut{S}$-conjugate 
to an element of the form $\phi i$, with $\phi$ a field automorphism and 
$i$ a graph automorphism, where $\phi$ and $i$ are commuting involutions 
(note that $\Phi_S\Gamma_S\simeq Z_2\times Z_2$). The action of 
$\Phi_S\Gamma_S$ on $M_c(S)$ is the same as on $\Outdiag{S}$. Thus, by 
\cite[Theorem 2.5.12 (g) and (i)]{GLS3}, both $\phi$ and $i$ invert 
$M_c(S)$, and therefore $u$ acts trivially on it. 
This argument shows that when $K$ is the universal covering group of $S$, 
no inversion on $Z(K)$ is induced by an automorphism of $K$. Assume now 
that $K$ is a covering of $S$ different from the universal one. 
%
%If $\abs{Z(K)}\not\in\set{12,24}$, then it is easy 
%to see that there exists an automorphism of $K$ inverting $Z(K)$. 
%For, i
%
If $3\nmid\abs{Z(K)}$ then $u$ inverts $Z(K)$. Assume therefore that
$3\vert\abs{Z(K)}$. If $Z(K)$ is cyclic of order $3$ or $6$, 
then $\phi$ inverts $Z(K)$. Otherwise we may argue as follows. Since $M_c(S)$ and $\Outdiag{S}$ 
are $\Aut{S}$-isomorphic, the elements of $\Out{S}$ that induce 
inversion on $M_c(S)$ are the six non central involutions, that is, the 
elements of the set   
$$T=\Outdiag{S}\phi\cup \Outdiag{S}i.$$
Note that from Proposition 6.2.2 and the proof of 
Theorem 6.3.1 in \cite{GLS3}, $\Outdiag{S}$ 
acts faithfully on the quotient group $M_e(S)/\Phi(M_e(S))$, and hence 
on $M_e(S)$.
Next, let $t$ be an element of $T$. Since $t$ inverts $\Outdiag{S}$,
$$C_{M_e(S)/\Phi(M_e(S))}(t)\simeq C_{\Phi(M_e(S))}(t)\neq \Phi(M_e(S)),$$
so $C_{M_e(S)}(t)=\seq{b}\simeq Z_4$, and $t$ inverts a unique cyclic 
subgroup of order $4$. This in particular shows that when $K$ is a 
covering extension with $Z(K)\simeq Z_2\times Z_2\times Z_3$ or 
$Z(K)\simeq Z_2\times Z_4\times Z_3$, 
%
%then
%$K$ is isomorphic to the quotient of the universal covering $\wt{K}$ 
%modulo its unique central elementary abelian subgroup of order $4$. In this 
%case NO inversion is induced ({\bf [GAP says it, prove why!]}).\\ 
no inversion is induced by automorphisms of $K$ on $Z(K)$, while 
if $Z(K)\simeq Z_4\times Z_3$, then inversion is induced.

% ({\bf [GAP says it, prove why!]}).\\
%Let $Z(K)\simeq Z_2\times Z_4\times Z_3$, then $K$ is isomorphic 
%to a quotient of $\wt{K}$ modulo a central sugroup of order $2$.
%Up to isomorphism $K$ is unique and NO inversion is induced ({\bf [GAP says it, prove why! Probably if it were then also in the case $Z(K)\simeq Z_2\times Z_2\times Z_3$ there should be]}).

Let $S=\PSU_4(3)$. Then $M_e(S)\simeq Z_3\times Z_3$ and
$M_c(S)\simeq Z_4$. Here $\Out{S}=\Outdiag{S}\Phi_S$ is 
isomorphic to a dihedral group of order $8$ acting faithfully 
on $M_e(S)$. In particular, the nontrivial central element of $\Out{S}$ 
is the unique element inducing inversion on $M_e(S)$. Note that this element 
belongs to $\Outdiag{S}$ and therefore it centralizes $M_c(S)$. 
This implies that no inversion can be induced on $M(S)$ 
by automorphisms of $S$. Therefore the universal covering group 
$\wt{S}$ has no automorphisms inverting its center. This argument 
can be extended to show that the same situation occours in any covering 
having center of order 12. However, for all other coverings $\widehat{S}$ 
of $S$ it can be easily checked that inversion on the 
center is induced either by the nontrivial central element of 
$\Outdiag{S}\Phi_S$, when $3$ divides  $\abs{Z(\widehat{S})}$, or 
by a field automorphism of order two, when $3\nmid \abs{Z(\widehat{S})}$. 

Let $S=\PSU_6(2)$, or $S=\null^2\EE_6(2)$. In both cases we have that 
$M_e(S)\simeq Z_2\times Z_2$, $M_c(S)\simeq Z_3$, and 
$\Out{S}\simeq S_3$ acts faithfully on $M_e(S)$ 
(see \cite[Proposition 6.2.2]{GLS3}). In particular, the 
trivial outer automorphism is the only one that inverts $M_e(S)$. 
Since it does not invert $M_c(S)$, the universal covering 
groups $\wt{S}$ % (for $S\in\set{\PSU_6(2), \null^2\EE_6(2)}$) 
have no automorphisms inverting their centers. On the contrary, every  
covering $\widehat{S}$ different from the universal one
possesses such automorphisms, which are either trivial if 
$3\nmid \abs{Z(\widehat{S})}$, or are field automorphisms. 
\end{proof}
For convenience, we write $\mathcal{L}$ for the set of quasisimple groups 
that appear as exceptions in Proposition \ref{pro:aut_quasisimple}; thus 
\begin{align*}
\mathcal{L}=&\left\{ \widehat{\PSL_3(4)}\right. \!\left(\!\textrm{with }    
                   \! Z\!\left(\widehat{\PSL_3(4)}\right)\!\geq Z_2\times Z_2\times Z_3\!\right)\!,\,
         \!          \!  \widehat{\PSU_4(3)}                     
                                        \left(\!\textrm{with }           
                   \! Z\!\left(\widehat{\PSU_4(3)}\right)\!\geq  Z_3\times Z_4\!\right)\!,  \\
&\,\,\,\, \left. \wt{\PSU_6(2)},\, \, \wt{\null^2\EE_6(2)}
 \right\}.
\end{align*}
\begin{rem}{\rm{According to 
\cite[Theorem 6.3.2]{GLS3}, when $S\simeq \PSL_3(4)$ (or 
when $S\simeq \PSU_4(3)$) there are precisely two non-isomorphic 
covering groups $\widehat{S}$ such that 
$\widehat{S}/Z(\widehat{S})\simeq S$ and 
$Z(\widehat{S})\simeq Z_2\times Z_4\times Z_3$ (respectively, 
$Z(\widehat{S})\simeq Z_3\times Z_4$). In all other cases in the list 
$\mathcal{L}$ the covering group of the associated finite simple group is unique. Therefore, up to isomorphism,
$\abs{\mathcal{L}}=9$.}}
\end{rem}

%\begin{align*}
%\mathcal{L}=\{ &\widehat{\PSL_3(4)} \textrm{ (with }    
%                    Z\left(\widehat{\PSL_3(4)}\right)\gtrsim Z_2\times Z_2\times Z_3),\, \,  \widehat{\PSU_4(3)} \textrm{ (with }           
%                    Z\left(\widehat{\PSU_4(3)}\right)\gtrsim Z_3\times Z_4),\\
%                & \wt{\PSU_6(2)},\, \, \wt{\null^2\EE_6(2)}\}. 
%\end{align*}

As an application of Proposition \ref{pro:aut_quasisimple} we 
prove the following:
\begin{cor}\label{cor:aut_semisimple}
Assume that $X$ is a finite semisimple group with all
components not in $\mathcal{L}$. Then 
there exists an automorphism of $X$ that inverts the 
center $Z(X)$.
\end{cor}
\begin{proof} 
Let $X$ be a finite semisimple group. Then $X$ is isomorphic 
to $D/N$, with $D=K_1\times K_2\times \ldots \times K_t$ 
the direct product of nonabelian quasisimple groups 
$K_i\not\in\mathcal{L}$ and 
$N\leq D$ such that $N\cap K_i=1$ for each $i\in\set{1,\ldots,t}$. Note that 
$N\leq Z(D)=Z(K_1)\times Z(K_2)\times \ldots \times Z(K_t)$.
By Proposition \ref{pro:aut_quasisimple}, for each $i\in\set{1,\ldots,t}$ 
there exists an automorphism $\alpha_i$ of $K_i$ that acts like the 
inversion on $Z(K_i)$. 
We may therefore define $\alpha\in\Aut{D}$ by 
$x^{\alpha}=(x_1,x_2,\ldots,x_t)^{\alpha}=(x_1^{\alpha_1},x_2^{\alpha_2},\ldots,x_t^{\alpha_t})$ 
for $x\in D$ (that is, where each $x_i\in K_i$).
Note that every element of $Z(D)$ is inverted by $\alpha$.
In particular $N^{\alpha}=N$, and $\alpha$ induces an 
automorphism on $D/N$ that inverts its center. 
\end{proof}

\section{The holomorph of a semisimple group}\label{sec:mainThm}
We first consider the case in which $G$ has no components belonging to $\mathcal{L}$.
\begin{pro}\label{pro:not_in_L}
Assume that $G$ is a finite nonabelian semisimple group and let 
$G=A_1A_2\ldots A_n$ be its unique central decomposition 
as a product of $\Aut{G}$-indecomposable factors. 
Suppose that the components of $G$ do not belong to $\mathcal{L}$. 
Then 
$$\mathcal{H}(G)=\set{G_J\vert J\subseteq I}.$$
Moreover, the group $T(G)=\mathrm{N}\Hol{G}/\Hol{G}$ is elementary 
abelian of order $2^n$.
\end{pro}
\begin{proof}
By Lemma \ref{lem:J}, we have that 
$\mathcal{J}(G)=\set{G_J\vert J\subseteq I}$, and 
therefore $\mathcal{H}(G)\subseteq\set{G_J\vert J\subseteq I}$.
We fix a subset $J$ of $I$ and, using Corollary \ref{cor:aut_semisimple}, 
we choose an automorphism $\alpha_{J^c}$ of $A_{J^c}$ that inverts 
$Z(A_{J^c})$. Define the map %$\varphi_J$ by 
$\varphi_{J}:  G \to G$ by $\varphi_J(x_Jx_{J^c})=x_J(x_{J^c})^{-\alpha_{J^c}}$.
Since $\alpha_{J^c}$ inverts $A_J\cap A_{J^c}\leq Z(G)$, the map 
$\varphi_J$ is a well-defined bijection of $G$, that is, an element of $S(G)$. 
We claim that the following hold: 
\begin{enumerate}
\item[(1)] $\varphi_J$ conjugates $G_I$ to $G_J$,% i.e. $(G_I)^{\varphi_J}=G_J$,
\item[(2)] $\varphi_J\in\mathrm{N}\Hol{G}$ and, if $J\neq I$, then $\varphi_J\not\in\Hol{G}$,
\item[(3)] $(\varphi_J)^2\in\Hol{G}$.
\end{enumerate}
Note that, by the arbitrary choice of $J\subseteq I$, 
once proved, (1) will imply that $\mathcal{H}(G)=\mathcal{J}(G)=\set{G_J\vert J\subseteq I}$, while (2), (3) and the fact that $T(G)$ acts regularly on $\mathcal{H}(G)$ will imply  that $T(G)$ is an elementary abelian $2$-group 
of rank $n$. \\

(1) $(G_I)^{\varphi_J}=G_J$.\\
We claim that for $x_J\in A_J$ and $x_{J^c}\in A_{J^c}$,
$$(\rho(x_Jx_{J^c}))^{\varphi_J}=\rho(x_J)
\lambda(x_{J^c}^{-\alpha_{J^c}}),$$   
or, equivalently, that: 
$$\rho(x_Jx_{J^c})\cdot\varphi_J=\varphi_J\cdot\rho(x_J)\lambda(x_{J^c}^{-\alpha_{J^c}}).$$
Let $g\in G$ and write $g$ as $g=y_Jy_{J^c}$ (with 
$y_J\in A_J$ and $y_{J^c}\in A_{J^c}$). Then
\begin{align*}
g^{\rho(x_Jx_{J^c})\cdot\varphi_J} &= (y_Jy_{J^c}x_Jx_{J^c})^{\varphi_J}=(y_Jx_Jy_{J^c}x_{J^c})^{\varphi_J}\\
    &= y_Jx_J(y_{J^c}x_{J^c})^{-\alpha_{J^c}}=y_Jx_Jx_{J^c}^{-\alpha_{J^c}}y_{J^c}^{-\alpha_{J^c}},
\end{align*} 
and 
\begin{align*}
g^{\varphi_J\cdot\rho(x_J)\lambda(x_{J^c}^{-\alpha_{J^c}})} 
&= (y_Jy_{J^c}^{-\alpha_{J^c}})^{\rho(x_J)
\lambda(x_{J^c}^{-\alpha_{J^c}})}= 
x_{J^c}^{-\alpha_{J^c}}y_Jy_{J^c}^{-\alpha_{J^c}}x_J\\
&=y_Jx_Jx_{J^c}^{-\alpha_{J^c}}y_{J^c}^{-\alpha_{J^c}}.
\end{align*} 
Therefore (1) is proved. Note that, together with Lemma \ref{lem:J}, 
we have proved that $\mathcal{J}(G)\subseteq \mathcal{H}(G)$ 
and therefore $\mathcal{H}(G)=\mathcal{I}(G)=\mathcal{J}(G)$.\\
 
(2) $\varphi_J\in\mathrm{N}\Hol{G}$.\\
By (1) we have that  
$$(N_{S(G)}(G_I))^{\varphi_J}=N_{S(G)}(G_I^{\varphi_J})=N_{S(G)}(G_J)
=N_{S(G)}(G_I),$$
since each $G_J$ lies in $\mathcal{H}(G)=\mathcal{I}(G)$, and 
therefore $N_{S(G)}(G_J)=N_{S(G)}(G_I)$; thus $\varphi_J\in\textrm{N}\Hol{G}$. \\
Furthermore, $G_J\neq G_I$ for $J\neq I$, thus we trivially have that 
$\varphi_J\not\in\Hol{G}$ for $J\neq I$.\\

(3) $(\varphi_J)^2\in\Hol{G}$.\\
We claim that for every $x_Jx_{J^c}\in G$:
$$(\rho(x_Jx_{J^c}))^{\varphi_J^2}=\rho(x_Jx_{J^c}^{\alpha_{J^c}^2}),$$
or, equivalently, that: 
$$\rho(x_Jx_{J^c})\cdot\varphi_J^2=\varphi_J^2\cdot
\rho(x_Jx_{J^c}^{\alpha_{J^c}^2}).$$
Let $g\in G$ and write $g=y_Jy_{J^c}$, 
where $y_J\in A_J$, $y_{J^c}\in A_{J^c}$. Then
\begin{align*}
g^{\rho(x_Jx_{J^c})\cdot\varphi_J^2} &= (y_Jy_{J^c}x_Jx_{J^c})^{\varphi_J^2}=(y_Jx_Jy_{J^c}x_{J^c})^{\varphi_J^2}\\
    &= y_Jx_J(y_{J^c}x_{J^c})^{\alpha_{J^c}^2}=
       y_Jy_{J^c}^{\alpha_{J^c}^2}x_Jx_{J^c}^{\alpha_{J^c}^2},
\end{align*} 
while 
\begin{align*}
g^{\varphi_J^2\cdot\rho(x_Jx_{J^c}^{\alpha_{J^c}^2})} &= (y_Jy_{J^c}^{\alpha_{J^c}^2})^{\rho(x_Jx_{J^c}^{\alpha_{J^c}^2})}= 
y_Jy_{J^c}^{\alpha_{J^c}^2}x_Jx_{J^c}^{\alpha_{J^c}^2}.
\end{align*} 
This completes the proof of Proposition \ref{pro:not_in_L}.
\end{proof}
\begin{rem}\label{rem:circ}
{\rm{For each fixed subset $J$ of $I$ we may define an 
operation $\circ_J$ on the set of elements of $G$, as follows:
$$g\circ_J h =g_Jg_{J^c}\circ_J h_Jh_{J^c} = g_Jh_Jh_{J^c}g_{J^c}$$
for each $g=g_Jg_{J^c}, h=h_Jh_{J^c}\in G$, where $g_J,h_J\in A_J$ and 
$g_{J^c}, h_{J^c}\in A_{J^c}$. Then $(G,\circ_J)$ is a group. 
Note that $\circ_I$ coincides with the group operation of $G$, 
while $\circ_{\varnothing}$ with the opposite multiplication in $G$, 
that is, $g_1\circ_{\varnothing} g_2=g_2g_1$, for every $g_1,g_2\in G$. 
With this notation, it is straightforward to prove that 
\begin{enumerate}
\item $(G,\circ_J)$ is a group isomorphic to $G_J$, 
for each $J\subseteq I$;
\item $\Aut{G}=\Aut{G,\circ_J}$ for each $J\subseteq I$ (see \cite[Theorem 5.2 (d)]{CDV};
\item if $G$ satisfies the assumptions of Proposition 
\ref{pro:not_in_L}, each map $\varphi_{J}$ is an isomorphism between 
$(G, \circ_I)$ and $(G, \circ_J)$. 
\end{enumerate}
}}
\end{rem}

We consider now the general situation in which exactly $l$ components of 
$G$ do belong to $\mathcal{L}$. If $K\in\mathcal{L}$ we call 
\emph{$\mathcal{L}$-critical} any subgroup $U$ of $Z(K)$ such that
respectively:
\begin{itemize}
\item $U\simeq Z_2\times Z_2\times Z_3$, if $K\simeq \widehat{\PSL_3(4)}$,
\item $U\simeq Z_3\times Z_4$, if $K\simeq \widehat{\PSU_4(3)}$ and \item $U=Z(K)$, if $K\simeq \wt{\PSU_6(2)}$ or if $K=\wt{\null^2\EE_6(2)}$.
\end{itemize}

\begin{thm}\label{thm:A}
Let $G$ be a finite nonabelian semisimple group and let 
$$G=A_1A_2\ldots A_n$$ be its unique central decomposition 
as a product of $\Aut{G}$-indecomposable factors. 
Assume that the number of factors $A_i$ of $G$ having components 
in $\mathcal{L}$ is exactly $l$, for some $0\leq l\leq n$. 
Then the set $\mathcal{H}(G)$ is completely known and it has cardinality
$\abs{\mathcal{H}(G)}= 2^h,$ 
where $m\leq h\leq n$ and $m=\min\set{n-l+1,n}$.\\ 
Also, the group $T(G)$ is an elementary abelian group of order $2^h$.\\
Moreover, if the centers of the factors $A_i$ are all amalgamated, 
then $\abs{\mathcal{H}(G)}= 2^m$ and $T(G)$ is elementary abelian of 
order $2^km$.
\end{thm}

\begin{proof}
By Proposition \ref{pro:not_in_L} the result is clear for $l=0$, so assume 
$l>0$.\\
Without loss of generality we may assume that $A_1,A_2,\ldots, A_l$ 
are the central $\Aut{G}$-indecomposable factors having components 
in $\mathcal{L}$. We set $L=\set{1,2,\ldots,l}$ and we claim that 
$\mathcal{H}(G)$ contains the set 
$$\mathcal{K}=\set{G_J, G_{J^c}\vert J\cap L=\varnothing},$$ 
whose cardinality is $2\cdot\abs{\mathcal{P}(L^c)}=2^m$.
By the definition of $\mathcal{H}(G)$ and  Lemmas \ref{lem:inv} and 
\ref{lem:G_J}, it is enough to show that $G_{J^c}\simeq G_I$ for 
each subset $J$ of $L^c$. By Corollary \ref{cor:aut_semisimple}
there exists an automorphism $\alpha_J$ of $A_J$ that inverts $Z(A_J)$.
Therefore the map $\varphi_{J^c}$ defined as in Proposition \ref{pro:not_in_L} by 
$\varphi_{J^c}(x_Jx_{J^c})=(x_J)^{-\alpha_{J}}x_{J^c}$
for each $x_J\in A_J$, $x_{J^c}\in A_{J^c}$, 
is a well-defined bijection of $G$ that conjugates $G_I$ to $G_{J^c}$.
Note that $T(G)$ contains the elementary abelian $2$-subgroup
$\set{\varphi_{J^c}\vert J\cap L=\varnothing}$.

Assume now that $\mathcal{K}\subset \mathcal{H}(G)$.\\ 
Note that $G_R\in \mathcal{H}(G)\setminus \mathcal{K}$ for some 
$R\subset I$ if and only if 
$G_{R^c}\in \mathcal{H}(G)\setminus \mathcal{K}$. This shows in 
particular that $\abs{\mathcal{H}(G)}$ is even.\\
Moreover,  by Remark \ref{rem:circ}, 
$G_R\in\mathcal{H}(G)\setminus\mathcal{K}$
if and only if $G_R$ is isomorphic to the group $(G,\circ_R)$.
Now, by Remark \ref{rem:circ}, any possible isomorphism $\alpha$ 
from $G$ to $(G, \circ_R)$ maps each $A_i$ to itself. 
In particular, if we take $r\in R\cap L$ and $s\in 
R^c\cap L$, 
the isomorphism $\alpha$ induces an automorphism on $A_r$ and 
an antihomomorphism on $A_s$, as  we have:
\begin{align}
(a_r b_r)^{\alpha} &= a_r^{\alpha} b_r^{\alpha}=
                   a_r^{\alpha}\circ_J b_r^{\alpha}  \quad \textrm{ for each } a_r,b_r\in A_r, \\
(a_sb_s)^{\alpha} &= a_s^{\alpha} b_s^{\alpha}= b_s^{\alpha} \circ_J a_s^{\alpha} \quad \textrm{ for each } a_s,b_s\in A_s.
\end{align}
Condition (3) implies that the restriction of $\alpha$ to 
$W=A_r\cap A_s$ is a homomorphism, while condition (4) 
implies that the restriction of $\alpha\circ \textrm{inv}$ to 
$W$ is a homomorphism. We deduce that the inversion map on $W$ is 
induced by an automorphism of $A_r$, which is in contradiction with 
Proposition \ref{pro:aut_quasisimple} if $W$ contains an 
$\mathcal{L}$-critical subgroup for some component of $A_r$. 
Thus we have proved that $G_R\in\mathcal{H}(G)\setminus\mathcal{K}$
if and only if for every $r\in R\cap L$ and every $s\in R^c\cap L$
the subgroup $A_r\cap A_s$ does not contain $\mathcal{L}$-critical 
subgroups of components. Note that this is equivalent to saying that 
the involutory map $\varphi_R$ is an element of $T(G)$. In particular, 
$T(G)$ is an elementary abelian $2$-group of order 
$\abs{\mathcal{H}(G)}$, which is therefore a power of $2$.\\
When $Z(A_i)=Z(G)$ for each $i=1,2,\ldots,n$, the result is clear since each $A_r\cap A_s=Z(G)$.
\end{proof}

\end{document}